\documentclass[a4paper,12pt]{article}

\usepackage[left=2cm,right=2cm, top=2cm,bottom=3cm,bindingoffset=0cm]{geometry}

\usepackage{verbatim}
\usepackage{amsmath}
\usepackage{amsthm}
\usepackage{amssymb}
\usepackage{delarray}
\usepackage{cite}
\usepackage{hyperref}
\usepackage{mathrsfs}
\usepackage{tikz}
\usetikzlibrary{patterns}
\usepackage{caption}
\DeclareCaptionLabelSeparator{dot}{. }
\newcommand{\al}{\alpha}
\newcommand{\be}{\beta}
\newcommand{\ga}{\gamma}
\newcommand{\de}{\delta}
\newcommand{\la}{\lambda}

\newcommand{\eps}{\varepsilon}

\theoremstyle{plain}

\numberwithin{equation}{section}

\newtheorem{thm}{Theorem}[section]
\newtheorem{lem}[thm]{Lemma}
\newtheorem{prop}[thm]{Proposition}
\newtheorem{cor}[thm]{Corollary}

\theoremstyle{definition}

\newtheorem{ip}[thm]{Inverse Problem}
\newtheorem{ap}[thm]{Auxiliary Problem}

\newtheorem{alg}[thm]{Algorithm}

\theoremstyle{remark}

\newtheorem{remark}[thm]{Remark}

\DeclareMathOperator{\sign}{sign}

 \allowdisplaybreaks

\begin{document}

\begin{center}
{\Large\bf Stability of the inverse Sturm-Liouville problem\\[0.2cm] on a graph with a cycle}
\\[0.5cm]
{\bf Natalia P. Bondarenko}
\end{center}

\vspace{0.5cm}

{\bf Abstract.} In this paper, the Sturm-Liouville operators on a graph with a cycle are considered. We study the inverse spectral problem, which consists in the recovery of the potentials from several spectra and a sequence of signs related to the quasiperiodic problem on the loop. The local stability of this inverse problem on the whole graph is proved. In addition, we investigate its uniform stability.

\medskip


\medskip


\section{Introduction} \label{sec:intr}

Consider the graph $G$ with a cycle presented in Figure~\ref{fig:graph}. The graph $G$ consists of the vertices $\{ v_j \}_0^m$ and the edges $\{ e_j \}_0^m$ ($m \ge 1$). Each edge $e_j$ has length $T_j > 0$ and joins the vertices $v_j$ and $v_0$ (i.e. $e_0$ is a loop). Thus $\{ v_j \}_1^m$ are boundary vertices and $v_0$ is the internal vertex.
We consider each edge $e_j$ as a segment parameterized by $x_j \in [0, T_j]$. The value $x_j = 0$ corresponds to the vertex $v_j$ and $x_j = T_j$, to $v_0$. 

\begin{figure}[h!]
\centering
\begin{tikzpicture}
\filldraw (0, 0) circle (2pt) node[anchor=east]{$v_0$};
\filldraw (3, 1.5) circle (2pt) node[anchor=west]{$v_1$};
\filldraw (3.5, 0.5) circle (2pt) node[anchor=west]{$v_2$};
\filldraw (3, -1.5) circle (2pt) node[anchor=west]{$v_m$};
\draw (0, 0) edge node[auto]{$e_1$} (3, 1.5);
\draw (0, 0) edge node[auto]{$e_2$} (3.5, 0.5);
\draw (0, 0) edge node[anchor=north]{$e_m$} (3, -1.5);
\filldraw (3.5, -0.5) circle (0.5pt);
\filldraw (3.5, -0.2) circle (0.5pt);
\filldraw (3.4, -0.8) circle (0.5pt);
\draw (-1, 0) circle (1);
\draw (-2.3, 0) node{$e_0$};
\end{tikzpicture}
\caption{Graph $G$}
\label{fig:graph}
\end{figure}

Consider the following system of the Sturm-Liouville equations on the graph $G$:
\begin{equation} \label{eqv}
-y_j''(x_j) + q_j(x_j) y_j(x_j) = \lambda y_j(x_j), \quad x_j \in (0, T_j), \quad j = \overline{0,m},
\end{equation}
with the boundary conditions
\begin{equation} \label{bc}
y_j(0) = 0, \quad j = \overline{1,m},
\end{equation}
and the matching conditions 
\begin{equation} \label{mc}
\begin{cases}
y_0(0) = a y_0(T_0) = y_j(T_j), \quad j = \overline{1,m}, \\
y_0'(0) = a^{-1} y_0'(T_0) + \sum\limits_{j = 1}^m y_j'(T_j),
\end{cases}
\end{equation}
where the so-called potentials $q_j$ belong to the corresponding classes $L_{2,\mathbb R}(0,T_j)$ ($j = \overline{0,m}$) of real-valued square-integrable functions, $\la$ is the spectral parameter, and $a \in \mathbb R \setminus \{ -1, 0, 1 \}$.

This paper is focused on an inverse spectral problem that consists in the recovery of the potentials $\{ q_j \}_0^m$ on the graph $G$ from spectral data. For this problem, we aim to investigate the stability, which, to the best of the author's knowledge, was not studied before.

The theory of inverse spectral problems has been most fully investigated for the Sturm-Liouville operators on intervals (see the monographs~\cite{Mar77, Lev84, FY01, Krav20} and references therein). For inverse problems on graphs, there is also an extensive literature. Differential operators on metric graphs, also called quantum graphs, model various processes in spacial networks. Such models appear in classical and quantum mechanics, organic chemistry, nanotechnology, waveguide theory, and other branches of science and engineering (see \cite{PPP04, BCFK06}). Inverse problems for quantum graphs started to be investigated by Gerasimenko and Pavlov~\cite{GP88}. Uniqueness of recovering the Sturm-Liouville potentials on arbitrary tree-graphs (i.e. graphs without cycles) from various types of spectral data has been proved in \cite{Bel04, BW05, Yur05}. The approach of Yurko~\cite{Yur05}, based on the method of spectral mappings \cite{Yur02}, allowed him to obtain a procedure for reconstruction of the potentials on tree-graphs and later on arbitrary compact graphs \cite{Yur10}. Subsequently, generalizations of that approach were developed for non-compact graphs (see, e.g., \cite{BF13, Ign15}). Various aspects of determining the graph structure by spectral characteristics were investigated in \cite{KN05, BV06, BW09, ALM10} and other studies. We separately mention the papers \cite{Yur08, Yur09, Kur10, Kur13, FI11, MT19, YB19, BS20, AE25} about inverse problems on graphs of special structure with a single cycle. More references to literature on inverse spectral theory of quantum graphs can be found in the surveys \cite{Yur16, Bel17, Bond23} and in the recent monographs \cite{MP20, Kur24}. 

The majority of studies on inverse spectral problems for differential operators on graphs are concerned with uniqueness and reconstruction. Stability was investigated much less, and there are only fragmentary results for graphs of specific structure in this direction. In particular, Bondarenko \cite{Bond18, Bond21-mmas} proved the local solvability and stability of partial inverse problems that consisted in recovering the Sturm-Liouville potential on a boundary edge of a graph from some part of the spectrum, while the potentials on the other edges were known a priori. Mochzuki and Trooshin \cite{MT19} have studied the conditional uniform stability of the inverse scattering problem on the lasso-shaped graph with an infinite boundary edge. The authors of \cite{MT19} relied on Marchenko's method \cite{Mar77}, and their result is limited to the potential on the boundary edge, while the transition to the internal loop would be the most challenging part of the problem. In \cite{Bond20-ipse, Bond25-matr}, inverse Sturm-Liouville problems on a compact star-shaped graph were investigated in the matrix form, and their stability was shown in that setting. In \cite{XB22}, several stability theorems were proved for the inverse scattering problem for the Schr\"odinger (Sturm-Liouville) operator on the half-line with the general self-adjoint boundary conditions. The results of \cite{XB22} can be directly applied to the star-shaped graph of infinite rays. We also mention the paper \cite{But23}, wherein the uniform stability of an inverse spectral problem is proved for a nonlocal functional-differential operator on a metric graph. However, nonlocal operators are fundamentally different from local (differential) ones. Inverse spectral analysis for these two classes of operators employs completely different methods.

Thus, up to now, there is no general approach to analyzing the stability of inverse problems for differential operators on graphs of arbitrary structure. At the same time, the recent studies \cite{KA23, AKK24-1, AKK24-2} present numerical methods for reconstruction of the Sturm-Liouville operators on graphs. The effectiveness of those methods in numerical examples, on the one hand, indirectly indicates the stability of the corresponding inverse problems and, on the other hand, motivates the analytical investigation of the stability.  

In this paper, we consider the reconstruction of the potentials $\{ q_j \}_0^m$ on the graph $G$ in Figure~\ref{fig:graph}. As spectral data, we use $(m+1)$ spectra generated by various boundary conditions and a sequence of signs related to the quasiperiodic problem on the loop (see Section~\ref{sec:main} for details). For convenience of working with the spectra, we construct the analytical characteristic functions, whose zeros coincide with the eigenvalues of the corresponding boundary value problems on the graph. We consider a small perturbation of the spectral data and prove the local stability of the inverse problem. Moreover, we study the uniform stability for potentials in a ball of a fixed (not small) radius. It is worth mentioning that fundamental results on the uniform stability of the direct and inverse Sturm-Liouville problems on a finite interval have been obtained by Savchuk and Shkalikov \cite{SS10}. A reader can find more references and a short overview on this topic in \cite{Bond25-matr}.

Our stability analysis is based on a constructive algorithm from \cite{Yur08} for solving the inverse problem. First, the potentials $\{q_j\}_1^m$ on the boundary edges are recovered. In contrast to \cite{Yur08}, we use the Gelfand-Levitan-type equations at this step, since the author finds them more convenient to investigation of the local stability. For studying the uniform stability, we apply the reconstruction formulas of the method of spectral mappings from \cite{Yur08}. Second, by using the boundary spectral data and the potentials on the boundary edges, we obtain the spectral data of the quasiperiodic problem on the loop. For the stability analysis at this step, we show that the characteristic functions of separate edges and of certain subgraphs are Lipschitz continuous with respect to the potentials. Third, we consider the quasiperiodic inverse Sturm-Liouville problem on the loop. That problem and its generalizations were investigated in \cite{Plak88, Yur12, Yur16-caot, Yur20, Bond25-quasi} and other studies. The stability result of \cite{Bond25-quasi} allows us to estimate the perturbation of the potential $q_0$ and thus to finish the proof of the main theorem (Theorem~\ref{thm:loc}). In the future, the technique of this paper can be extended to other classes of quantum graphs, in particular, to arbitrary trees.

The paper is organized as follows. Section~\ref{sec:main} contains the main results and outlines the method of the proofs. In Section~\ref{sec:prelim}, we provide auxiliary representations for the characteristic functions, analyze their dependence on the potentials $\{ q_j \}_0^m$, and present other preliminaries. In Section~\ref{sec:bound}, the stability of the reconstruction of the potentials on the boundary edges is proved. In Section~\ref{sec:transit}, we study the stability of the transition from the boundary edges to the loop. Section~\ref{sec:quasi} contains the necessary information about the quasiperiodic inverse problem and the proof of the theorem on the local stability of the inverse problem on the graph (Theorem~\ref{thm:loc}). In Section~\ref{sec:uni}, we discuss the uniform stability of the inverse problem.

\smallskip

Throughout the paper, we use the following notation:
\begin{itemize}
\item $\mathbf{q} := \{ q_j \}_0^m$, $\mathbf{T} := \sum\limits_{j = 0}^m T_j$.
\item $\mathcal L = \mathcal L(\mathbf{q})$ is the boundary value problem \eqref{eqv}--\eqref{mc}.
\item $\sign a = \begin{cases}
1, & a > 0, \\
0, & a = 0, \\
-1, & a < 0.
\end{cases}$
\item $\la = \rho^2$, $\mu = \theta^2$.
\item $L_{2,\mathbb R}^0(0,T) := \left\{ q \in L_{2,\mathbb R}(0,T) \colon \int_0^T q(t) \, dt = 0\right\}$.
\item $L_{2,\mathbb R}^0(G) := \left\{ \mathbf{q} = \{ q_j \}_0^m \colon q_j \in L_{2,\mathbb R}^0(0,T_j), \, j = \overline{0,m} \right\}$.
\item Along with $\mathcal L$, we consider a problem $\tilde{\mathcal L} = \mathcal L(\tilde{\mathbf q})$ of the same form \eqref{eqv}--\eqref{mc} but with different potentials $\tilde{\mathbf{q}} = \{ \tilde q_j \}_0^m$, $\tilde q_j \in L_{2,\mathbb R}(0, T_j)$. The parameters $m$, $\{ T_j \}_0^m$, and $a$ for the problems $\mathcal L$ and $\tilde{\mathcal L}$ are assumed to be the same. We agree that, if a symbol $\be$ denotes an object related to $\mathcal L$, then the symbol $\tilde \be$ with tilde will denote the analogous object related to $\tilde{\mathcal L}$, and $\hat \be := \be - \tilde \be$.
\item In estimates, the symbol $C$ is used for various positive constants that do not depend on $x$, $\la$, etc. In proofs, we mean that constants $C$ depend on the same parameters as in the formulation of the corresponding theorem or lemma.
\item For $T > 0$, we denote by $PW(T)$ the class of the Paley-Wiener functions of the form
\begin{equation} \label{PW}
\mathcal F(\rho) = \int_{-T}^T f(t) \exp(i \rho t) \, dt, \quad f \in L_2(-T, T).
\end{equation}
Plancherel's theorem implies
\begin{equation} \label{plan}
\| \mathcal F \|_{L_2(\mathbb R)} = \sqrt{2 \pi} \| f \|_{L_2(-T,T)}.
\end{equation}
\item For $Q > 0$, introduce the set 
\begin{equation} \label{BQ}
B_Q := \left\{ \mathbf{q} \in L_{2,\mathbb R}^0(G) \colon \| q_j \|_{L_2(0,T_j)} \le Q, \, j = \overline{0,m} \right\}.
\end{equation}

\item The notation $f(\rho) \asymp g(\rho)$ means that
$$
\forall \rho \quad c |f(\rho)| \le |g(\rho)| \le C |f(\rho)|,
$$
where $c$ and $C$ are positive constants.
\end{itemize}

\section{Main results} \label{sec:main}

For $k = \overline{1,m}$, denote by $\mathcal L_k = \mathcal L_k(\mathbf{q})$ the boundary value problem for equations \eqref{eqv} with the matching conditions \eqref{mc} and the boundary conditions
\begin{equation} \label{bck}
y_k'(0) = 0, \quad y_j(0) = 0, \quad j = \overline{1,m} \setminus k.
\end{equation}

The problems $\mathcal L$ and $\mathcal L_k$ ($k = \overline{1,m}$) are self-adjoint and have purely discrete spectra $\Lambda$ and $\Lambda_k$ ($k = \overline{1,m}$), respectively, which consist of real eigenvalues. For reconstruction of the potentials $\{ q_j \}_0^m$,
we will use the spectra $\Lambda$ and $\Lambda_k$ ($k = \overline{1,m}$) together with some additional information related to the quasiperiodic inverse problem on the loop $e_0$.

For each $j = \overline{0,m}$, denote by $C_j(x_j, \la)$ and $S_j(x_j, \la)$ the solutions of equation \eqref{eqv} satisfying the initial conditions
$$
S_j(0,\la) = C_j'(0,\la) = 0, \quad S_j'(0,\la) = C_j(0,\la) = 1.
$$

For each fixed $x_j \in [0,T_j]$, the functions $C_j^{(\nu)}(x_j, \la)$ and $S_j^{(\nu)}(x_j, \la)$ ($\nu = 0, 1$) are entire in $\la$. Denote by $\{ \la_n \}_1^{\infty}$ the zeros of the function $h(\la) := S_0(T_0,\la)$ and put
\begin{equation} \label{defsi}
H(\la) := a C_0(T_0,\la) - a^{-1} S'_0(T_0,\la), \quad
\sigma_n := \sign H(\la_n), \quad n \ge 1.
\end{equation}

Consider the following inverse spectral problem on the graph $G$:

\begin{ip} \label{ip:eigen}
Given the spectra $\Lambda$, $\Lambda_k$ ($k = \overline{1,m}$) and the signs $\{ \sigma_n \}_1^{\infty}$, find the potentials $\{ q_j \}_0^m$.
\end{ip}

In this paper, we focus on the stability of recovering the potentials. Therefore, the edge lengths $\{ T_j \}_0^m$ and the parameter $a$ are assumed to be known and fixed. Anyway, the reconstruction of this kind of parameters from the spectral data can also be discussed (see, e.g., \cite{ALM10}). 

The eigenvalues of the Sturm-Liouville problems on graphs can be found as the zeros of some entire functions, which are called characteristic functions. In particular, the spectra $\Lambda$ and $\Lambda_k$ ($k = \overline{1,m}$) coincide with the zeros of the characteristic functions $\Delta(\la)$ and $\Delta_k(\la)$ ($k = \overline{1,m}$), respectively, defined as follows (see \cite{Yur08}):
\begin{align} \label{defDelta}
& \Delta(\la) := (d(\la) - 2) \Delta^{\Pi}(\la) + a h(\la) \Delta^K(\la), \\ \label{defDeltak}
& \Delta_k(\la) := (d(\la) - 2) \Delta_k^{\Pi}(\la) + a h(\la) \Delta_k^K(\la), \quad k = \overline{1,m},
\end{align}
where $d(\la)$ is the so-called Hill-type discriminant related to the quasiperiodic problem (see \cite{Bond25-quasi}):
\begin{equation} \label{defd}
d(\la) := a C_0(T_0,\la) + a^{-1} S_0'(T_0,\la) 
\end{equation}
and
\begin{align} \label{defDPK1}
& \Delta^{\Pi}(\la) := \prod_{j = 1}^m S_j(T_j,\la), \quad \Delta^K(\la) := \sum_{r = 1}^m S_r'(T_r, \la) \prod_{\substack{j = 1 \\ j \ne r}}^m S_j(T_j,\la), \\ \label{defDPK2}
& \Delta_k^{\Pi}(\la) := C_k(T_k,\la) \prod_{\substack{j = 1 \\ j \ne k}}^m S_j(T_j,\la), \\ \label{defDPK3} &\Delta_k^K(\la) := C_k'(T_k,\la) \prod_{\substack{j = 1 \\ j \ne k}}^m S_j(T_j,\la) + 
C_k(T_k,\la) \sum_{\substack{r = 1 \\ r \ne k}}^m S_r'(T_r, \la) \prod_{\substack{j = 1 \\ j \ne r, \, j \ne k}}^m S_j(T_j,\la).
\end{align}

It is worth mentioning that the functions participating in the right-hand sides of \eqref{defDelta}--\eqref{defDeltak} can be interpreted as characteristic functions of some eigenvalue problems:
\begin{itemize}
\item The zeros of $d(\la)$ coincide with the eigenvalues of the quasiperiodic problem for equation \eqref{eqv} ($j = 0$) with the boundary conditions
\begin{equation} \label{bcquasi}
    y_0(0) = a y_0(T_0), \quad y_0'(0) = a^{-1} y_0'(T_0).
\end{equation}
\item The zeros of $h(\la) = S_0(T_0,\la)$ coincide with the eigenvalues of the Sturm-Liouville problem for equation \eqref{eqv} ($j = 0$) with the Dirichlet boundary conditions $y_0(0) = y_0(T_0) = 0$.
\item The zeros of $\Delta^{\Pi}(\la)$ coincide with the eigenvalues of the Sturm-Liouville system \eqref{eqv} for $j = \overline{1,m}$ on the graph of $m$ separate edges $\{ e_j \}_1^m$ with the Dirichlet boundary conditions \eqref{bc} and
\begin{equation} \label{bcT}
y_j(T_j) = 0, \quad j = \overline{1,m}.
\end{equation}
\item The zeros of $\Delta^K(\la)$ coincide with the eigenvalues of the Sturm-Liouville system \eqref{eqv} for $j = \overline{1,m}$ on the star-shaped graph consisting of the edges $\{ e_j \}_1^m$ joined at the vertex $v_0$ with the Dirichlet boundary conditions \eqref{bc} and the standard (Kirchhoff) matching conditions
\begin{equation} \label{Kirch}
y_1(T_1) = y_j(T_j), \quad j = \overline{2,m}, \qquad \sum_{j = 1}^m y_j'(T_j) = 0.
\end{equation}
\item For $k \in \{ 1, 2, \dots, m \}$, the zeros of $\Delta^{\Pi}_k(\la)$ and $\Delta^K_k(\la)$ coincide with the eigenvalues of the Sturm-Liouville problems \eqref{eqv} ($j = \overline{1,m}$), \eqref{bck}, \eqref{bcT} and \eqref{eqv} ($j = \overline{1,m}$), \eqref{bck}, \eqref{Kirch}, respectively.
\end{itemize}

In the above, the coincidence of eigenvalues and zeros of entire functions is understood with multiplicities. The upper indices $\Pi$ and $K$ mean ``product'' and ``Kirchhoff'', respectively.

It was shown in \cite{Yur08} that the characteristic functions $\Delta(\la)$ and $\Delta_k(\la)$ ($k = \overline{1,m}$) can be uniquely recovered their zeros (i.e. from the corresponding eigenvalues) as infinite products by Hadamard's Factorization Theorem. Consequently, Inverse Problem~\ref{ip:eigen} can be equivalently reformulated as follows:

\begin{ip} \label{ip:main}
Given the characteristic functions $\Delta(\la)$, $\Delta_k(\la)$ ($k = \overline{1,m}$) and the signs $\{ \sigma_n \}_1^{\infty}$, find the potentials $\{ q_j \}_0^m$.    
\end{ip}

Note that the eigenvalues in the spectra $\Lambda$ and $\Lambda_k$ ($k = \overline{1,m}$) can be multiple and/or arbitrarily close to one another. Therefore, it will be more convenient for us to study the stability of Inverse Problem~\ref{ip:main} in terms of the characteristic functions.

Denote by $\mathcal L^0$ and $\mathcal L_k^0$ ($k = \overline{1,m}$) the eigenvalue problems $\mathcal L$ and $\mathcal L_k$ ($k = \overline{1,m}$), respectively, with the zero potentials: $q_j^0 \equiv 0$, $j = \overline{0,m}$. The corresponding solutions and characteristic functions can be found explicitly:
\begin{align} \nonumber
S_j^0(x_j, \la) & = \frac{\sin \rho x_j}{\rho}, \quad C_j^0(x_j, \la) = \cos \rho x_j, \\ \label{defDelta0}
\Delta^0(\la) & = \bigl( (a + a^{-1}) \cos \rho T_0 - 2 \bigr) \prod_{j = 1}^m \frac{\sin \rho T_j}{\rho} + a \sum_{r = 1}^m \cos \rho T_r \prod_{\substack{j = 0 \\ j \ne r}}^m \frac{\sin \rho T_j}{\rho}, \\ \nonumber
\Delta^0_k(\la) & = \bigl( (a + a^{-1}) \cos \rho T_0 - 2 \bigr) \cos \rho T_k \prod_{\substack{j = 1 \\ j \ne k}}^m \frac{\sin \rho T_j}{\rho} \\ \label{defDelta0k} & + a \left( - \rho \sin \rho T_k \prod_{\substack{j = 0 \\ j \ne k}}^m \frac{\sin \rho T_j}{\rho} + \cos \rho T_k \sum_{\substack{r = 1 \\ r \ne k}}^m \cos \rho T_r \prod_{\substack{j = 0 \\ j \ne r, j \ne k}}^m \frac{\sin \rho T_j}{\rho} \right), \quad k = \overline{1,m},
\end{align}
where $\la = \rho^2$.

For technical simplicity, assume that $\mathbf{q} \in L_{2,\mathbb R}^0(G)$. Then, in Section~\ref{sec:prelim}, we get the following representations:
\begin{align} \label{asymptDelta}
& \Delta(\la) = \Delta^0(\la) + \frac{\varkappa(\rho)}{\rho^{m+1}}, \\ \label{asymptDeltak}
& \Delta_k(\la) = \Delta_k^0(\la) + \frac{\varkappa_k(\rho)}{\rho^m}, \quad k = \overline{1,m},
\end{align}
where $\varkappa(\rho)$ and $\varkappa_k(\rho)$ are functions of the class $PW(\mathbf{T})$, $\mathbf{T} := \sum\limits_{j = 0}^m T_j$. 

By considering small perturbations of the remainder terms in \eqref{asymptDelta} and \eqref{asymptDeltak}, we obtain the following theorem on the local stability of Inverse Problem~\ref{ip:main}:

\begin{thm} \label{thm:loc}
Let $\mathcal L = \mathcal L(\mathbf{q})$ be a boundary value problem of form \eqref{eqv}--\eqref{mc} with $\mathbf{q} \in L^0_{2,\mathbb R}(G)$. Then, there exists $\eps > 0$ (depending on $\mathcal L$) such that, for any $\tilde{\mathbf{q}} \in L_{2,\mathbb R}^0(G)$ satisfying the requirements $\sigma_n = \tilde \sigma_n$ for all $n \ge 1$ and
\begin{equation} \label{difDe}
\de := \bigl\| \rho^{m+1} (\Delta - \tilde \Delta)(\rho^2) \|_{L_2(\mathbb R)} + 
\sum_{k = 1}^m \bigl\| \rho^m (\Delta_k - \tilde \Delta_k)(\rho^2)\bigr\|_{L_2(\mathbb R)} \le \eps,
\end{equation}
the following estimates hold:
\begin{equation} \label{estqj}
\| q_j - \tilde q_j\|_{L_2(0, T_j)} \le C \de, \quad j = \overline{0,m},
\end{equation}
where $C$ is a constant depending only on $\mathcal L$.
\end{thm}

According to Remark~\ref{rem:nonemp} below, for each fixed problem $\mathcal L$ and every $\eps > 0$, there exist infinitely many potential vectors $\tilde{\mathbf q}$ satisfying the requirements in Theorem~\ref{thm:loc}, so the estimates \eqref{estqj} hold for a non-empty set of potentials.

The proof of Theorem~\ref{thm:loc} is based on the following constructive algorithm for solving Inverse Problem~\ref{ip:main}:

\begin{alg} \label{alg:main}
Suppose that the characteristic functions $\Delta(\la)$, $\Delta_k(\la)$ ($k = \overline{1,m}$), and the signs $\{ \sigma_n \}_{n \ge 1}$ are given. We have to find the potentials $\{ q_j \}_0^m$.   
\begin{enumerate}
\item For each $k = \overline{1,m}$, recover $q_k$ from $\Delta(\la)$ and $\Delta_k(\la)$.
\item Using $\{ q_j \}_1^m$, $\Delta(\la)$, and $\Delta_1(\la)$, find the Hill-type discriminant $d(\la)$ and the Dirichlet eigenvalues $\{ \la_n \}_1^{\infty}$ (i.e. the zeros of $h(\la)$).
\item Solve the quasiperiodic problem on the loop: given $d(\la)$, $\{ \la_n \}_1^{\infty}$, and $\{ \sigma_n \}_1^{\infty}$, recover $q_0$.
\end{enumerate}
\end{alg}

The basic strategy of Algorithm~\ref{alg:main} coincides with \cite[Algorithm~2]{Yur08}. However, the implementation details of steps 1--2 of our algorithm are different, since we focus on proving the stability of the reconstruction. In particular, for the recovery of each potential $q_k$ at step~1, we use the Gelfand-Levitan-type equation, which is derived in Section~\ref{sec:bound}, in contrast to the method of spectral mappings in \cite{Yur08}. Our crucial idea at step~2 consists in considering the linear system \eqref{defDelta}--\eqref{defDeltak} at specific points $\{ \nu_n \}$ that generate a Riesz basis $\{ \exp(i \nu_n t) \}$ in $L_2(-T_0, T_0)$ (see Section~\ref{sec:transit} for details). At step~3, we rely on the results of \cite{Bond25-quasi} for the quasiperiodic problem.

Theorem~\ref{thm:loc} has a local nature. Additionally, we study the uniform stability of the inverse problem for $\mathbf{q}$ in the ball $B_Q$ defined by \eqref{BQ} and obtain the following result.

\begin{thm} \label{thm:uni}
Let $Q > 0$ be fixed and $a \in \mathbb R \setminus \{ 0\}$. Then, for any $\mathbf{q}$ and $\tilde{\mathbf{q}}$ in $B_Q$, there holds
\begin{equation} \label{quni}
\| q_k - \tilde q_k \|_{L_2(0,T_k)} \le C \left( \| \rho^{m+1} (\Delta - \tilde \Delta)(\rho^2) \|_{L_2(\mathbb R)} + \| \rho^m (\Delta_k - \tilde \Delta_k)(\rho^2) \|_{L_2(\mathbb R)} \right)
\end{equation}
for $k = \overline{1,m}$, where the constant $C$ depends only on $Q$.
\end{thm}

Note that the cases $a = \pm 1$ correspond to the periodic and antiperiodic problems on the loop (see the boundary conditions \eqref{bcquasi}). In Theorem~\ref{thm:loc}, we exclude these cases, since we rely on the quasiperiodic inverse spectral theory from \cite{Bond25-quasi}, which is specific for $a \ne \pm 1$. However, a number of results of this paper, especially those concerning the boundary edges (in particular, Theorem~\ref{thm:uni}), are also valid for $a = \pm 1$.

\section{Preliminaries} \label{sec:prelim}

In this section, we provide auxiliary representations of the characteristic functions that participate in formulas \eqref{defDelta} and \eqref{defDeltak}. Furthermore, we establish the Lipschitz continuity of those functions with respect to the corresponding potentials and present other preliminaries.

\begin{prop}[\hspace*{-3pt}\cite{Mar77}, proof of Lemma~1.3.2] \label{prop:SC}
Under the condition $q_j \in L_{2,\mathbb R}^0(0,T_j)$, the following relations hold:
\begin{align} \label{intS}
& S_j(T_j,\la) = \frac{\sin \rho T_j}{\rho} + \frac{\kappa_{s,j}(\rho)}{\rho^2}, \quad S_j'(T_j,\la) = \cos \rho T_j + \frac{\eta_{s,j}(\rho)}{\rho}, \\ \nonumber
& C_j(T_j,\la) = \cos \rho T_j + \frac{\kappa_{c,j}(\rho)}{\rho}, \quad
C_j'(T_j,\la) = -\rho \sin \rho T_j + \eta_{c,j}(\rho),
\end{align}
where $\kappa_{s,j}$, $\eta_{s,j}$, $\kappa_{c,j}$, and $\eta_{c,j}$ are functions of $PW(T_j)$, $j = \overline{0,m}$.
\end{prop}

In particular, the relations \eqref{intS} imply the following standard asymptotics as $|\rho| \to \infty$ (see, e.g., \cite[Section~1.1]{FY01}):
\begin{align} \label{asymptS}
& S_j(x, \la) = \frac{\sin \rho x}{\rho} + O\bigl( |\rho|^{-2} \exp(|\mbox{Im} \, \rho| x)\bigr), \\ \label{asymptSp}
& S'_j(x, \la) = \cos \rho x + O\bigl( |\rho|^{-1} \exp(|\mbox{Im}\rho| x)\bigr).
\end{align}

\begin{prop}[\hspace*{-3pt}\cite{FY01}, Section 1.1] \label{prop:StL}
Let $j \in \{ 0, 1, \dots, m \}$ be fixed and $q_j \in L_{2,\mathbb R}^0(0,T_j)$. Then, the following assertions hold: 

(i) The spectrum Sturm-Liouville problem for equation \eqref{eqv} with the boundary conditions $y_j(0) = y_j(T_j) = 0$ is a countable set of real and simple eigenvalues $\{ \la_{n,j} \}_{n = 1}^{\infty}$, which coincide with the zeros of the characteristic function $S_j(T_j,\la)$ and have the asymptotics
\begin{equation} \label{asymptlanj}
\sqrt{\la_{n,j}} = \frac{\pi n}{T_j} + \frac{\varsigma_{n,j}}{n}, \quad \{ \varsigma_{n,j} \} \in l_2, \quad n \ge 1.
\end{equation}

(ii) For every $Q > 0$ and $\| q_j \|_{L_2(0,T_j)} \le Q$, there holds $\| \{ \varsigma_{n,j} \}_{n = 1}^{\infty} \|_{l_2} \le C$, where the constant $C$ depends only on $Q$. In particular, the spectrum $\{ \la_{n,j} \}_{n = 1}^{\infty}$ is bounded from below by a constant that depends only on $Q$.

(iii) The eigenfunctions $\{ S_j(x_j, \la_{n,j}) \}_{n = 1}^{\infty}$ form an orthogonal basis in $L_2(0, T_j)$.

(iv) The zeros of the functions $S_j'(T_j,\la)$, $C_j(T_j,\la)$, and $C_j'(T_j,\la)$ are real, simple, and coincide with the eigenvalues of the Sturm-Liouville problems for equation \eqref{eqv} under the boundary conditions $y_j(0) = y_j'(T_j) = 0$, $y_j'(0) = y_j(T_j) = 0$, and $y_j'(0) = y_j'(T_j) = 0$, respectively.
\end{prop}

Substituting the representations of Proposition~\ref{prop:SC} into the relations \eqref{defd}--\eqref{defDPK3}, we get the following corollary.

\begin{cor} \label{cor:DPK}
The following relations hold:
\begin{align} \label{intdPW}
& d(\la) = (a + a^{-1}) \cos \rho T_0 + \frac{\varkappa^d(\rho)}{\rho}, \\ \nonumber
& \Delta^{\Pi}(\la) = \rho^{-m} \prod_{j = 1}^m \sin \rho T_j + \frac{\varkappa^{\Pi}(\rho)}{\rho^{m+1}}, \\ \nonumber
& \Delta^K(\la) = \rho^{-(m-1)} \sum_{r = 1}^m \cos \rho T_r \prod_{\substack{j = 1 \\ j \ne r}}^m \sin \rho T_j + \frac{\varkappa^K(\rho)}{\rho^m}, \\ \nonumber
& \Delta_k^{\Pi}(\la) = \rho^{-(m-1)} \cos \rho T_k \prod_{\substack{j = 1 \\ j \ne k}}^m \sin \rho T_j + \frac{\varkappa_k^{\Pi}(\rho)}{\rho^m}, \\ \nonumber
& \Delta_k^K(\la) = \rho^{-(m-2)} \left( -\sin \rho T_k \prod_{\substack{j = 1 \\ j \ne k}}^m \sin \rho T_j + \cos \rho T_k \sum_{\substack{r = 1 \\ r \ne k}}^m \cos \rho T_r \prod_{\substack{j = 1 \\ j \ne r, \, j \ne k}}^m \sin \rho T_j \right) + \frac{\varkappa_k^K(\rho)}{\rho^{m-1}},
\end{align}
where $\varkappa^d \in PW(T_0)$, the functions $\varkappa^{\Pi}$, $\varkappa^K$, $\varkappa_k^{\Pi}$, and $\varkappa_k^K$ belong to $PW(\mathbf{T} - T_0)$, $k = \overline{1,m}$.
\end{cor}

Substituting the representations of Proposition~\ref{prop:SC} and Corollary~\ref{cor:DPK} into \eqref{defDelta} and \eqref{defDeltak}, we arrive at the relations \eqref{asymptDelta} and \eqref{asymptDeltak}, respectively. 

Proceed to the stability of the representations in Proposition~\ref{prop:SC} and Corollary~\ref{cor:DPK}.

\begin{lem} \label{lem:stabSC}
Let $Q > 0$ and $j \in \{ 0, 1, 2, \dots, m \}$ be fixed.
Then, for any $q_j$ and $\tilde q_j$ in $L_{2, \mathbb R}^0(0,T_j)$ satisfying the conditions 
\begin{equation} \label{ballQ}
\| q_j \|_{L_2(0,T_j)} \le Q, \quad \| \tilde q_j \|_{L_2(0,T_j)} \le Q, 
\end{equation}
the following estimates hold:
\begin{align} \label{stabSj}
& \| \rho^2 \hat S_j(T_j, \rho^2) \|_{L_2(\mathbb R)} \le C \| \hat q_j \|_{L_2(0, T_j)}, \quad
\| \rho \hat S_j'(T_j, \rho^2) \|_{L_2(\mathbb R)} \le C \| \hat q_j \|_{L_2(0, T_j)}, \\ \label{stabCj}
& \| \rho \hat C_j(T_j, \rho^2) \|_{L_2(\mathbb R)} \le C \| \hat q_j \|_{L_2(0, T_j)}, \quad
\| \hat C'_j(T_j, \rho^2) \|_{L_2(\mathbb R)} \le C \| \hat q_j \|_{L_2(0, T_j)},
\end{align}
where the constant $C$ depends only on $Q$ and $j$.
\end{lem}

\begin{proof}
According to \eqref{intS}, we have
$$
\rho^2 \hat S_j(T_j,\rho^2) = \int_0^{T_j} \hat K_j(t) \cos \rho t \, dt, \quad \hat K_j \in L_{2,\mathbb R}(0,T_j).
$$
Under the hypothesis of this lemma, Theorem~7 in \cite{But22} implies
\begin{equation} \label{estKsi}
\| \hat K_j \|_{L_2(0,T_j)} \le C \| \{ \hat\varsigma_{n,j} \}_{n = 1}^{\infty} \|_{l_2},
\end{equation}
where $\{ \varsigma_{n,j} \}$ are the remainder terms from the asymptotics \eqref{asymptlanj}. We may assume that the constant $C$ in \eqref{estKsi} depends only on $Q$, taking part (ii) of Proposition~\ref{prop:StL} into account.
Furthermore, the estimates (2.13) of Theorem 2.12 in \cite{SS10} yield
\begin{equation} \label{estsiq}
\| \{ \hat \varsigma_{n,j} \}_{n = 1}^{\infty} \|_{l_2} \le C \| \hat q_j \|_{L_2(0,T_j)}.
\end{equation}
Combining \eqref{estKsi} and \eqref{estsiq}, we arrive at the estimate
\begin{equation} \label{estKq}
\| \hat K_j \|_{L_2(0,T_j)} \le C \| \hat q_j \|_{L_2(0,T_j)} 
\end{equation}
in the ball \eqref{ballQ}. Obviously, \eqref{estKq} implies \eqref{stabSj} for $\hat S_j$.

Note that the results of \cite{But22} are valid for a wide variety of entire functions, including $S_j'(T_j,\la)$, $C_j(T_j,\la)$, and $C_j'(T_j,\la)$. Theorem 2.12 in \cite{SS10} provides the estimates analogous to \eqref{estsiq} for the Dirichlet-Neumann spectrum (and symmetrically for the Neumann-Dirichlet one). Moreover, the results of \cite{SS10} have been transferred to the Neumann-Neumann case in \cite{Bond25-HN} by using the Darboux-type transforms. In view of part (iv) of Proposition~\ref{prop:StL}, the described approach implies all the four estimates \eqref{stabSj}--\eqref{stabCj}. 

It is worth mentioning that the uniform estimate \eqref{estKq} has been directly derived in \cite{But21} (see formula (72) for the case  $M \equiv 0$ in \cite{But21}). 
\end{proof}

Using Proposition~\ref{prop:SC}, Lemma~\ref{lem:stabSC}, and the definitions \eqref{defDPK1}--\eqref{defDPK3}, we obtain the following corollary.

\begin{cor} \label{cor:stabDPK}
Suppose that $Q > 0$ and, for $j = \overline{1,m}$, functions $q_j$ and $\tilde q_j$ of $L_{2,\mathbb R}^0(0,T_j)$ satisfy \eqref{ballQ}. Then
\begin{align*} 
& \| \rho^{m+1} \hat \Delta^{\Pi}(\rho^2) \|_{L_2(\mathbb R)} + 
\| \rho^m \hat \Delta^K(\rho^2) \|_{L_2(\mathbb R)} \le C \sum_{j = 1}^m \| \hat q_j \|_{L_2(0,T_j)}, \\ \nonumber
& \| \rho^m \hat \Delta_k^{\Pi}(\rho^2) \|_{L_2(\mathbb R)} + \| \rho^{m-1} \hat \Delta_k^K(\rho^2) \|_{L_2(\mathbb R)} \le C \sum_{j = 1}^m \| \hat q_j \|_{L_2(0,T_j)}, \quad k = \overline{1,m},
\end{align*}
where the constant $C$ depends only on $Q$. 
\end{cor}

Finally, using the relations \eqref{defDelta}--\eqref{defd} together with Lemma~\ref{lem:stabSC} and Corollary~\ref{cor:stabDPK}, we obtain the following result.

\begin{cor} \label{cor:stabDelta}
Suppose that $Q > 0$, $a \in \mathbb R \setminus \{ 0 \}$, $\mathbf{q}$ and $\tilde{\mathbf{q}}$ in $B_Q$.
Then
\begin{equation*}
\| \rho^{m + 1} \hat \Delta(\rho^2) \|_{L_2(\mathbb R)} + \sum_{k = 1}^m \| \rho^m \hat \Delta_k(\rho^2) \|_{L_2(\mathbb R)} \le C \sum_{j = 1}^m \| \hat q_j \|_{L_2(0, T_j)},
\end{equation*}
where the constant $C$ depends only on $Q$. In other words, 
the remainders $\varkappa(\rho)$ and $\varkappa_k(\rho)$ ($k = \overline{1,m}$) in the asymptotics \eqref{asymptDelta} and \eqref{asymptDeltak}, respectively, are Lischitz continuous in $L_2(\mathbb R)$-norms with respect to $\mathbf{q}$ in $B_Q$.

In particular, $\| \varkappa \|_{L_2(\mathbb R)}$ and $\| \varkappa_k \|_{L_2(\mathbb R)}$ $(k = \overline{1,m})$ are bounded and so the spectrum $\Lambda$ is bounded from below by a constant that depends only on $Q$.
\end{cor}

\begin{remark} \label{rem:nonemp}
Arguments of this section show the feasibility of Theorem~\ref{thm:loc}. Indeed, due to Corollary~\ref{cor:stabDelta}, the remainders $\varkappa(\rho)$ and $\varkappa_k(\rho)$ in \eqref{asymptDelta} and \eqref{asymptDeltak}, respectively, continuously depend on $\mathbf{q}$ in the corresponding $L_2$-norms. Consequently, for every $\eps > 0$, there exists a sufficiently small neighborhood of $\mathbf{q}$ in $L_{2,\mathbb R}^0(G)$ such that the inequality \eqref{difDe} holds for any $\tilde{\mathbf{q}}$ in that neighborhood. Furthermore, using Proposition~\ref{prop:StL}, Lemma~\ref{lem:stabSC}, and the definitions \eqref{defsi}, we conclude that the values $H(\la_n)$ depend continuously on $q_0$, so the signs $\sigma_n = \pm 1$ are preserved for small perturbations of $q_0$. In the case $\sigma_n = 0$, the equality $d(\la_n) = \pm 2$ holds, and it produces some restriction on $q_0$ (see \cite[Remark~5.2]{Bond25-quasi} for details). Anyway, in the quasiperiodic case $a \ne \pm 1$, there is only a finite number of zero signs. Consequently, there exists an infinite set of potentials $\tilde q_0$ such that $\sigma_n = \tilde \sigma_n$ for all $n \ge 1$ in every $\eps$-neighborhood of $q_0$ in $L_{2,\mathbb R}^0(0,T_0)$. Thus, for every $\eps > 0$, there exist potentials $\tilde{\mathbf{q}}$ satisfying the conditions of Theorem~\ref{thm:loc}.  
\end{remark}

\section{Reconstruction on boundary edges} \label{sec:bound}

In this section, we prove the local stability of step~1 in Algorithm~\ref{alg:main}. Namely, we fix $k \in \{ 1, 2, \dots, m \}$ and consider the following auxiliary inverse problem.

\begin{ip} \label{ip:loc}
Given the characteristic functions $\Delta(\la)$ and $\Delta_k(\la)$, find the potential~$q_k$.
\end{ip}

Inverse Problem~\ref{ip:loc} is analogous to the classical Borg problem \cite{Borg46} of the potential reconstruction from two spectra. However, the principal difference of Inverse Problem~\ref{ip:loc} is that the given spectra contain information from the whole graph, not only from the edge $e_k$. In this section, we obtain the Gelfand-Levitan-type equation \eqref{GL} for Inverse Problem~\ref{ip:loc}. Moreover, we prove the solvability of equation \eqref{GL} under a small perturbation of the spectral data and estimate its solution.

Introduce the function 
\begin{equation} \label{defM}
M_k(\la) := -\dfrac{\Delta_k(\la)}{\Delta(\la)},
\end{equation}
which is called the Weyl function with respect to the boundary vertex $v_k$ due to \cite{Yur08}. Clearly, the Weyl function is meromorphic in the $\la$-plane and its poles coincide with the eigenvalues of the problem $\mathcal L$, so they are real and bounded from below.
According to \cite[Theorem~2]{Yur08}, the Weyl function $M_k(\la)$ uniquely specifies the potential $q_k$. Hence, the solution of Inverse Problem~\ref{ip:loc} is unique. 

Below in this section, we for brevity omit the index $k$ when possible. In particular, we put $q := q_k$, $M := M_k$, $T := T_k$, etc. 
Consider two problems $\mathcal L$ and $\tilde{\mathcal L}$ of form \eqref{eqv}--\eqref{mc} with different potentials $\mathbf{q}$ and $\tilde{\mathbf{q}}$, respectively, both in $L_{2,\mathbb R}^0(G)$.

Introduce the following contour in the $\rho$-plane:
\begin{equation} \label{defga}
\ga := \{ \rho = \sigma + i \tau \colon -\infty < \sigma < \infty \},
\end{equation}
and its image in the $\la$-plane:
\begin{equation} \label{defGa}
\Gamma := \{ \la = \rho^2 \colon \rho \in \gamma \}
\end{equation}
with the counter-clockwise circuit (see Figure~\ref{fig:contour}). The value $\tau > 0$ is chosen so large that all the eigenvalues of $\mathcal L$ and $\tilde{\mathcal L}$ lie inside $\Gamma$. This is possible since the spectra $\Lambda$ and $\tilde \Lambda$ are bounded from below by virtue of Corollary~\ref{cor:stabDelta}.

\begin{figure}[h!]
\centering
\begin{tikzpicture}
  \draw[thin,dashed] (-0.5, 0) edge (6.5, 0);
  \draw[very thick,domain=-2.5:-1.5] plot (\x*\x,\x);
  \draw[very thick,domain=-1.5:1.5,<-] plot (\x*\x,\x);
  \draw[very thick,domain=1.5:2.5,<-] plot (\x*\x,\x);
  \filldraw (0.5, 0) circle(2pt);
  \filldraw (2.1, 0) circle(2pt);
  \filldraw (3.7, 0) circle(2pt);
  \filldraw (4.2, 0) circle (2pt);
  \filldraw (5.4, 0) circle (2pt);
  \filldraw (6.3, 0) circle (2pt);
  \draw (4,-0.4) node{eigenvalues};
\end{tikzpicture}
\caption{Contour $\Gamma$}
\label{fig:contour}
\end{figure}
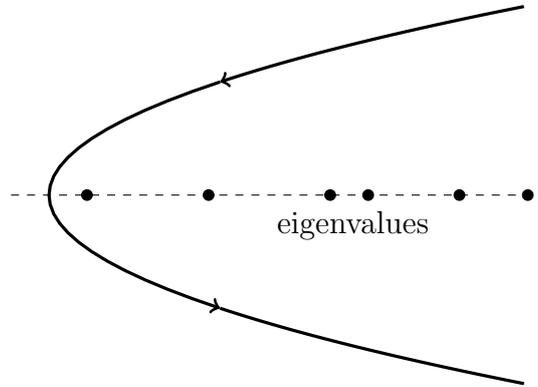

According to the results of \cite{Yur08}, the following relation holds:
\begin{equation} \label{maineq}
\tilde S(x, \la) = S(x, \la) + \frac{1}{2 \pi i} \int_{\Gamma} R(x, \la, \mu) \tilde S(x, \mu) \, d\mu,
\end{equation}
where 
\begin{equation} \label{defR}
R(x, \la, \mu) := \int_0^x S(t, \la) S(t, \mu) \hat M(\mu) \, dt, \quad \hat M(\mu) := M(\mu) - \tilde M(\mu).
\end{equation}

The contour \eqref{defGa} is slightly different from the one in \cite{Yur08}, but this difference is not principal, since in the both cases the contour encloses all the eigenvalues of $\mathcal L$ and $\tilde{\mathcal L}$. 

For the solutions $S(x, \la)$ and $\tilde S(x, \la)$, there exists the transformation operator (see, e.g., \cite[Chapter~I]{Lev84}):
\begin{equation} \label{trans}
    \tilde S(x, \la) = S(x, \la) + \int_0^x K(x, t) S(t, \la) \, dt,
\end{equation}
where the kernel $K(x, t)$ is continuous for $0 \le t \le x \le T$, the derivatives $K_x$ and $K_t$ are $L_2$-functions for each fixed $x \in (0,T]$ or $t \in [0,T)$. Furthermore, $K(x, 0) = 0$ and
$$
K(x, x) = \frac{1}{2}\int_0^x (\tilde q - q)(t) \, dt.
$$

Our analysis of the inverse problem on a boundary edge relies on the following theorem.

\begin{thm} \label{thm:GL}
For each fixed $x \in [0,T]$, the kernel $K(x, t)$ of the transformation operator in \eqref{trans} satisfies the Gelfand-Levitan-type equation
\begin{equation} \label{GL}
K(x, t) + F(x, t) + \int_0^x K(x, s) F(s, t) \, ds = 0,
\end{equation}
where
\begin{equation} \label{defF}
F(x, t) := -\frac{1}{2 \pi i} \int_{\Gamma} \hat M(\mu) S(x, \mu) S(t, \mu) \, d\mu.
\end{equation}
\end{thm}

Note that equation \eqref{GL} differs from the classical Gelfand-Levitan equation (see \cite{Lev84, FY01}), since the function \eqref{defF} in \eqref{GL} depends on $\hat M(\mu)$ and so contains the spectral information from the whole graph.

Before proving Theorem~\ref{thm:GL}, we need two auxiliary lemmas.

\begin{lem} \label{lem:hatM}
$\hat M(\rho^2) \in L_2(\gamma)$.
\end{lem}

\begin{proof}
In view of \eqref{defDelta0}--\eqref{asymptDeltak}, we have
\begin{equation} \label{est2side}
\Delta(\rho^2) \asymp \rho^{-m}, \quad |\Delta_k(\rho^2)| \le C |\rho|^{-(m-1)}, \quad \rho \in \gamma.
\end{equation}
The similar estimates hold for $\tilde \Delta(\rho^2)$ and $\tilde \Delta_k(\rho^2)$.

The relations \eqref{asymptDelta} and \eqref{asymptDeltak} imply that
$\rho^{m+1} \hat \Delta(\rho^2)$ and $\rho^m \hat \Delta_k(\rho^2)$ are functions of $PW(\mathbf{T})$, in particular, they belong to $L_2(\gamma)$. Consequently, using \eqref{defM} and \eqref{est2side}, we arrive at the assertion of the lemma.
\end{proof}

\begin{lem} \label{lem:conv}
The following assertions hold:

(i) The integral in \eqref{maineq} converges absolutely and uniformly by $x \in [0,T]$ and $\la$ on compact sets. 

(ii) The integral in \eqref{defF} converges absolutely and uniformly by $(x,t)$ on $[0,T] \times [0, T]$. 
\end{lem}

\begin{proof}
Using the standard asymptotics \eqref{asymptS} and \eqref{defR}, we obtain
$$
\bigl| R(x, \la, \mu) \tilde S(x, \mu) \bigr| \le C |\theta|^{-2} |\hat M(\theta^2)|, \quad \mu = \theta^2,
$$
for $x \in [0,T]$ and $\la$ on compact sets.
Then, taking the relation $d\mu = 2 \theta d\theta$ and Lemma~\ref{lem:hatM} into account, we arrive at the assertion (i) of this lemma. The assertion (ii) is proved similarly by using \eqref{defF}.
\end{proof}

\begin{proof}[Proof of Theorem~\ref{thm:GL}]
Fix $x \in (0,T]$. Substitution of \eqref{trans} into \eqref{maineq} and the change of the integration order by taking Lemma~\ref{lem:conv} into account imply
\begin{equation} \label{sm1}
\int_0^x K(x, t) S(t, \la) \, dt = \int_0^x S(t, \la) \left( \frac{1}{2\pi i} \int_{\Gamma} S(t,\mu) \hat M(\mu) \Bigl[ S(x,\mu) + \int_0^x K(x,s) S(s,\mu) \, ds \Bigr] d\mu \right) \, dt.
\end{equation}

In view of part (iii) in Proposition~\ref{prop:StL},
the equation
$$
\int_0^x f(t) S(t,\la) \, dt \equiv 0, \quad \forall \la \in \mathbb C,
$$
only has the zero solution $f$ in $L_2(0,x)$. Thus, changing the order of the two inner integrals in \eqref{sm1}, we derive
$$
K(x,t) = \frac{1}{2\pi i} \int_{\Gamma} \hat M(\mu) S(x,\mu) S(t,\mu) \, d\mu + \int_0^x K(x,s) \left( \frac{1}{2\pi i} \int_{\Gamma} \hat M(\mu) S(x,\mu) S(s,\mu) \, d\mu \right) ds,
$$
which together with \eqref{defF} imply \eqref{GL}. 
\end{proof}

The following theorem asserts the local solvability and stability of the Gelfand-Levitan-type equation \eqref{GL}.

\begin{thm} \label{thm:GLloc}
Let $\mathcal L = \mathcal L(\mathbf{q})$ be a boundary value problem of form \eqref{eqv}--\eqref{mc} with $\mathbf{q} \in L^0_{2,\mathbb R}(G)$. Then, there exists $\eps > 0$ (depending on $\mathcal L$) such that, for any entire functions $\tilde \Delta(\la)$ and $\tilde \Delta_k(\la)$ satisfying the conditions
\begin{gather} \label{PWde}
\rho^{m+1}\hat \Delta(\rho^2) \in PW(\mathbf{T}), \quad
\rho^m \hat \Delta_k(\rho^2) \in PW(\mathbf{T}), \\ \label{defdek}
\de_k := \| \rho^{m+1}\hat \Delta(\rho^2) \|_{L_2(\mathbb R)} + 
\| \rho^m \hat \Delta_k(\rho^2) \|_{L_2(\mathbb R)} \le \eps,
\end{gather}
the corresponding equation \eqref{GL} has a unique solution $K(x,t)$ for each fixed $x \in [0,T]$, there holds $\dfrac{d}{dx} K(x,x) \in L_2(0,T)$, and the function $\tilde S(x, \la)$ defined by \eqref{trans} solves the initial value problem
\begin{equation} \label{initS}
-\tilde S''(x, \la) + \tilde q(x) \tilde S(x, \la) = \la \tilde S(x, \la), \quad \tilde S(0,\la) = 0, \quad \tilde S'(0,\la) = 1,
\end{equation}
where 
\begin{equation} \label{findq}
\tilde q(x) := q(x) + 2 \frac{d}{dx} K(x,x).
\end{equation}
Moreover,
\begin{equation} \label{estq}
\| \hat q \|_{L_2(0,T)} \le C \de_k,
\end{equation}
where the constant $C$ depends only on the problem $\mathcal L$. In this theorem, the function $F(x,t)$ in equation \eqref{GL} is assumed to be constructed by using $\tilde \Delta(\la)$ and $\tilde \Delta_k(\la)$ via \eqref{defF}, wherein the contour $\Gamma$ is given by \eqref{defGa} and $\tau > 0$ is chosen so large that the eigenvalues of $\mathcal L$ lie inside $\Gamma$.
\end{thm}

\begin{proof}
Let $\mathcal L$ and $\ga$ satisfy the hypothesis of this theorem. Suppose that the functions $\tilde \Delta(\la)$ and $\tilde \Delta_k(\la)$ fulfill \eqref{PWde}.
Then $\tilde \Delta(\la)$ and $\tilde \Delta_k(\la)$ satisfy the asymptotics \eqref{asymptDelta} and \eqref{asymptDeltak}, respectively. Furthermore, according to \eqref{PW}, we have
$$
\| \mathcal F \|_{L_2(\mathbb R)} \asymp \| \mathcal F \|_{L_2(\ga)} \asymp \| f \|_{L_2(-T,T)}, \quad \forall \mathcal F \in PW(T).
$$
Consequently, it follows from \eqref{difDe} that
\begin{equation} \label{difDe2}
\| \rho^{m+1}\hat \Delta(\rho^2) \|_{L_2(\ga)} \le C\de_k, \quad
\| \rho^m \hat \Delta_k(\rho^2) \|_{L_2(\ga)} \le C\de_k. 
\end{equation}

Next, assume that \eqref{difDe} holds for a sufficiently small $\eps > 0$. Then, using \eqref{difDe2} and following the proof strategy of Lemma~\ref{lem:hatM}, we conclude that $\hat M(\rho^2) \in L_2(\ga)$ and 
\begin{equation} \label{estM}
\| \hat M(\rho^2) \|_{L_2(\ga)} \le C \de_k.
\end{equation}

Using \eqref{defF}, \eqref{asymptS}, and \eqref{estM}, we obtain
\begin{equation} \label{estF}
\max_{x,t} |F(x,t)| \le C \de_k.
\end{equation}
Furthermore, differentiating \eqref{defF} and using the standard asymptotics \eqref{asymptS} and \eqref{asymptSp},
we obtain
\begin{align*}
F_x(x,t) & = -\frac{1}{2\pi i} \int_{\Gamma} \hat M(\mu) S'(x,\mu) S(t,\mu) \, d\mu \\
& = -\frac{1}{\pi i} \int_{\ga} \hat M(\theta^2) \cos \theta x \sin \theta t \, d\theta + \int_{\ga} \hat M(\theta^2) O(\theta^{-1}) d\theta.
\end{align*}
Consequently, the derivative $F_x$ and, similarly, $F_t$ belong to $L_2(0,T)$ for each fixed $x$ or $t$. In view of \eqref{estM}, the $L_2$-norms of the derivatives are estimated by $C\de_k$, in particular:
\begin{gather} \label{estFx}
\| F_x(x,.) \|_{L_2(0,x)} \le C\de_k, \quad \| F_t(.,t) \|_{L_2(0,t)} \le C\de_k, \\ \label{estFx2}
\left\| \frac{d}{dx} F(x,x) \right\|_{L_2(0,T)} \le C\de_k.
\end{gather}

For each fixed $x \in [0,T]$, consider the integral operator $I + \mathfrak F$ on $C[0,x]$, where $I$ is the identity operator and 
$(\mathfrak F z)(t) := \int_0^x z(s) F(s,t) \, ds$.
In view of \eqref{estF}, for sufficiently small $\de_k$, the operator $I + \mathfrak F$ has a bounded inverse operator on $C[0,x]$. Consequently, equation \eqref{GL} has a unique continuous solution $K(x,.)$. Repeating the standard arguments (see, e.g., \cite[Section 1.5.2]{FY01}), we show that the solution $K(x,t)$ possesses the same smoothness as $F(x,t)$. In particular, the function $\tilde q(x)$ defined by \eqref{findq} belongs to $L_2(0,T)$ and the relations \eqref{initS} hold.

Let us focus on obtaining the estimate \eqref{estq}. It follows from \eqref{GL}, \eqref{estF}, and the boundedness of the operator $(I + \mathfrak F)^{-1}$ that
\begin{equation} \label{estK}
\max_{x,t} |K(x,t)| \le C \de_k.
\end{equation}
Differentiation of \eqref{GL} by $x$ implies
\begin{equation} \label{GLx}
K_x(x,t) + \int_0^x K_x(x,s) F(s,t) \, ds = -F_x(x,t) - K(x,x) F(x,t).
\end{equation}
According to \eqref{estF}, \eqref{estFx}, and \eqref{estK}, the norm of the right-hand side in $L_2(0,x)$ is estimated by $C \de_k$. Considering $I + \mathfrak F$ as a boundedly invertible operator on $L_2(0,x)$, we conclude from \eqref{GLx} that
\begin{equation} \label{estKx}
\| K_x(x,.) \|_{L_2(0,x)} \le C \de_k, \quad x \in [0,T].
\end{equation}

Next, putting $t = x$ in \eqref{GL} and differentiating by $x$, we derive
\begin{equation} \label{GLx2}
\frac{d}{dx} K(x,x) + \frac{d}{dx} F(x,x) + K(x,x) F(x,x) + \int_0^x K_x(x,s) F(s,x) \, ds + \int_0^x K(x,s) F_x(s,x) \, ds = 0.
\end{equation}
Using \eqref{estF}--\eqref{estK} and \eqref{estKx}, we deduce from \eqref{GLx2} that
$$
\left\| \frac{d}{dx} K(x,x) \right\|_{L_2(0,T)} \le C \de_k,
$$
which together with \eqref{findq} readily implies \eqref{estq}.
\end{proof}

Note that the assertion of Theorem~\ref{thm:GLloc} in some aspects is more general than Theorem~\ref{thm:loc}. Specifically, the existence of $\tilde q$ is asserted in Theorem~\ref{thm:GLloc} but required in Theorem~\ref{thm:loc}. Furthermore, in Theorem~\ref{thm:GLloc}, the function $\tilde q$ is not supposed to be real-valued. At the same time, Theorem~\ref{thm:GLloc} does not answer the question whether $\tilde \Delta(\la)$ and $\tilde \Delta_k(\la)$ are the characteristic functions of some boundary value problems on the graph. In addition, it is not guaranteed that $\int\limits_0^T \tilde q(t) \, dt = 0$.
Anyway, under the requirements of Theorem~\ref{thm:loc}, Theorem~\ref{thm:GLloc} immediately yields the estimates \eqref{estqj} for $j = \overline{1,m}$. Thus, it remains to prove \eqref{estqj} for $j = 0$.

\section{Transition through the vertex} \label{sec:transit}

In this section, the stability of step~2 in Algorithm~\ref{alg:main} is investigated. Specifically, we use the recovered potentials $\{ q_j \}_1^m$ on the boundary edges together with the initially given characteristic functions for finding the spectral data related to the loop.

Without loss of generality, assume that $T_0 = 1$.
According to \eqref{intS} and \eqref{intdPW}, the functions $d(\la)$ and $h(\la)$ admit the representations
\begin{align} \label{intd}
& d(\la) = (a + a^{-1}) \cos \rho + \frac{1}{\rho} \int_0^1 D(t) \sin \rho t \, dt, \\ \label{inth}
& h(\la) = \frac{\sin \rho}{\rho} + \frac{1}{\rho^2} \int_0^1 K(t) \cos \rho t \, dt,
\end{align}
where $D, K \in L_{2,\mathbb R}(0,1)$. 
Since the function $h(\la)$ is analytic at zero, then $\int_0^1 K(t) \, dt = 0$.
Recall that $\{ \la_n \}_1^{\infty}$ are the zeros of $h(\la)$.

In this section, we focus on the following problem.

\begin{ap} \label{ap:loop}
Given $\{ q_j \}_1^{\infty}$, $\Delta(\la)$, and $\Delta_1(\la)$, find $D(t)$ and $\{ \la_n \}_1^{\infty}$.
\end{ap}

The functions $d(\la)$ and $h(\la)$ can be found from \eqref{defDelta} and \eqref{defDeltak} for $k = 1$ by Cramer's rule
\begin{equation} \label{cram}
d(\la) = 2 + \frac{E_1(\la)}{E(\la)}, \quad h(\la) = \frac{E_2(\la)}{a E(\la)},
\end{equation}
where
\begin{align} \label{defE}
& E(\la) := \Delta^{\Pi}(\la) \Delta_1^K(\la) - \Delta_1^{\Pi}(\la) \Delta^K(\la), \\ \label{defE1}
& E_1(\la) := \Delta(\la) \Delta_1^K(\la) - \Delta_1(\la) \Delta^K(\la), \\ \nonumber
& E_2(\la) := \Delta^{\Pi}(\la) \Delta_1(\la) - \Delta_1^{\Pi}(\la) \Delta(\la).
\end{align}

The relation \eqref{defE} can be simplified by using \eqref{defDPK1}--\eqref{defDPK3}:
\begin{equation} \label{Esimp}
E(\la) = -\left( \prod_{j = 2}^m S_j(T_j,\la) \right)^2.
\end{equation}
Hence, the zeros of $E(\la)$ coincide with the (doubled) zeros of $S_j(T_j,\la)$, which are real.

In order to proceed from $d(\la)$ and $h(\la)$ to the functions $D(t)$ and $K(t)$, respectively, it will be convenient to choose a sequence $\{ \nu_n \}_{-\infty}^{\infty}$ that $\{ \exp(i \nu_n t) \}_{-\infty}^{\infty}$ is a Riesz basis in $L_2(-1,1)$ and $E(\nu_n^2) \ne 0$, $n \in \mathbb Z$. For definiteness, fix $\al > 0$ and put
\begin{equation} \label{defnu}
\nu_n := \pi n + i \alpha, \quad n \in \mathbb Z.
\end{equation}

For a function $f \in L_2(-1,1)$, denote by $\{ \widehat f_n \}_{-\infty}^{\infty}$ the corresponding Fourier coefficients:
$$
\widehat f_n := \int_{-1}^1 f(t) \exp(i \nu_n t) \, dt.
$$

Due to the Riesz-basis property of the sequence $\{ \exp(i \nu_n t) \}_{-\infty}^{\infty}$, the following two-side inequality holds:
\begin{equation} \label{twoside}
\| f \|_{L_2(-1,1)} \asymp \| \{ \widehat f_n \}_{-\infty}^{\infty} \|_{l_2}, \quad \forall f \in L_2(-1,1).
\end{equation}

Additionally, we need the following two auxiliary lemmas about exponential Riesz bases.

\begin{lem} \label{lem:bess}
Suppose that $a, b, c > 0$.
Let $\{ \exp(i \rho_n x) \}$ be a Riesz basis in $L_2(-a,a)$
and $|\mbox{Im} \, \rho_n| \le c$. Then 
\begin{equation} \label{bess}
\forall f \in L_2(-b,b) \quad
\sum_n \left| \int_{-b}^b f(x) \exp(i \rho_n x) dx\right|^2 \le C \int_{-b}^b |f(x)|^2 \, dx.
\end{equation}
\end{lem}

\begin{proof}
For $a = b$, the inequality \eqref{bess} is a well-known property of a Riesz basis. For $b < a$, the function $f$ can be completed by zero to the interval $(-a,a)$.

If $b > a$, then without loss of generality assume that $b = ka$, $k \in \mathbb N$. Thus
\begin{align*}
\sum_n \left| \int_{-ka}^{ka} f(x) \exp(i \rho_n x)\right|^2 & \le
\sum_n k \sum_{j = 1}^k \left| \int_{-a}^a f(x+s_j) \exp(i \rho_n (x + s_j)) dx\right|^2 \\ & \le C \sum_{j = 1}^m \int_{-a}^a |f(x+s_j)|^2 dx = C \int_{-ka}^{ka} |f(x)|^2 \, dx,
\end{align*}
where $s_j := (2j-k-1)a$.
\end{proof}

\begin{lem} \label{lem:PW}
Let $\{ \nu_n \}_{-\infty}^{\infty}$ be given by \eqref{defnu}. Then
$$
\forall \mathcal F \in PW(T) \quad \| \{ \mathcal F(\nu_n) \}_{-\infty}^{\infty} \|_{l_2} \le C \| \mathcal F \|_{L_2(\mathbb R)},
$$
where the constant $C$ depends only on $\al$ and $T$ (possibly $T \ne 1$).
\end{lem}

\begin{proof}
Recall that any Paley-Wiener function $\mathcal F(\rho)$ is represented by \eqref{PW} with $f \in L_2(-T,T)$ and $\{ \exp(i \nu_n t) \}_{-\infty}^{\infty}$ is a Riesz basis in $L_2(-1,1)$. Then, by Lemma~\ref{lem:PW}, we get
$$
\| \{ \mathcal F(\nu_n) \}_{-\infty}^{\infty} \|_{l_2} \le C \| f \|_{L_2(-T,T)},
$$
which together with \eqref{plan} conclude the proof.
\end{proof}

Next, let us formulate and prove two lemmas on the local stability of Auxiliary Problem~\ref{ap:loop}.

\begin{lem} \label{lem:stabaux}
Let $\mathcal L = \mathcal L(\mathbf{q})$ be a boundary value problem of form \eqref{eqv}--\eqref{mc} with $\mathbf{q} \in L_{2,\mathbb R}^0(G)$. Then, there exists $\eps > 0$ (depending on $\mathcal L$) such that, for any $\tilde{\mathbf q} \in L_{2,\mathbb R}^0(G)$ satisfying
\begin{equation} \label{esteps}
\epsilon := \| \rho^{m+1} \hat \Delta(\rho^2) \|_{L_2(\mathbb R)} + 
\| \rho^m \hat \Delta_1(\rho^2) \|_{L_2(\mathbb R)} + \sum_{j = 1}^m \| \hat q_j \|_{L_2(0,T_j)} \le \eps,
\end{equation}
there holds
\begin{equation} \label{estDK}
\| \hat D \|_{L_2(0,1)} + \| \hat K \|_{L_2(0,1)} \le C \epsilon,
\end{equation}
where the constant $C$ depends only on $\mathcal L$.
\end{lem}

\begin{proof}
Suppose that $\mathcal L$ and $\tilde{\mathbf{q}}$ satisfy the hypothesis of this lemma with some $\eps \le 1$. 
Note that
$$
\int_0^1 D(t) \sin \rho t \, dt = \int_{-1}^1 \mathscr D(t) \exp(i \rho t) \, dt,  
$$
where
$$
\mathscr D(t) := \frac{1}{2 i}\begin{cases} D(t), \quad t > 0, \\
                                        -D(-t), \quad t < 0,
                                        \end{cases},
\quad \mathscr D \in L_2(-1,1).                                       
$$

Consider the Fourier coefficients
$$
\widehat{\mathscr D}_n = \int_{-1}^1 \mathscr D(t) \exp(i \nu_n t) \, dt, \quad
\widehat{\tilde{\mathscr D}}_n = \int_{-1}^1 \tilde{\mathscr D}(t) \exp(i \nu_n t) \, dt, \quad n \in \mathbb Z.
$$

Denote $\mu_n := \nu_n^2$.
By virtue of \eqref{intd} and \eqref{cram}, we have
\begin{equation} \label{relDn}
\widehat{\mathscr D}_n - \widehat{\tilde{\mathscr D}}_n = \nu_n \left( \frac{E_1(\mu_n)}{E(\mu_n)} - \frac{\tilde E_1(\mu_n)}{\tilde E(\mu_n)}\right).
\end{equation}

Let us estimate the right-hand side of \eqref{relDn}. 

\smallskip

\textit{Estimates for $E(\mu_n)$}. According to parts (i) and (ii) of Proposition~\ref{prop:StL}, the zeros of the functions
$S_j(T_j,\la)$ ($j = \overline{1,m}$) are real and bounded from below. If $\al > 0$ in \eqref{defnu} is large enough, then $S_j(T_j,\mu_n) \ne 0$ for $n \in \mathbb Z$, $j = \overline{1,m}$. Taking the asymptotics \eqref{asymptS} into account, we conclude that
\begin{equation} \label{estSj}
S_j(T_j, \mu_n) \asymp \nu_n^{-1}, \quad n \in \mathbb Z, \quad j = \overline{1,m}.
\end{equation}

Substituting \eqref{estSj} into \eqref{Esimp}, we get
\begin{equation} \label{estE}
E(\mu_n) \asymp \nu_n^{-2(m-1)}, \quad n \in \mathbb Z.
\end{equation}
If $\eps$ is small enough, then the two-sided estimate similar to \eqref{estE} also holds for $\tilde E(\mu_n)$. Since $\{ q_j \}_1^m$ are fixed and $\{ \tilde q_j \}_1^m$ satisfy \eqref{esteps} with $\eps \le 1$, then the inequalities \eqref{ballQ} hold for some $Q > 0$ (depending only on $\mathbf{q}$) and $j = \overline{1,m}$. Consequently, Lemmas~\ref{lem:stabSC} and~\ref{lem:PW} imply
\begin{equation} \label{estSjn}
\bigl\| \{ \nu_n^2 \hat S_j(T_j,\mu_n) \}_{-\infty}^{\infty} \bigr\|_{l_2} \le C \| \rho^2 \hat S_j(T_j, \rho^2) \|_{L_2(\mathbb R)} \le C \| \hat q_j \|_{L_2(0,T_j)}.
\end{equation}

Using \eqref{Esimp} and \eqref{estSj}, we obtain
$$
|\hat E(\mu_n)| \le C |\nu_n|^{-(2m-3)} \sum_{j = 2}^m |\hat S(T_j,\mu_n)|.
$$
Taking \eqref{estSjn} and \eqref{esteps} into account, we deduce
\begin{equation} \label{estEn}
\| \{ \nu^{2m-1} \hat E(\mu_n) \}_{-\infty}^{\infty} \|_{l_2} \le C \epsilon.
\end{equation}

\smallskip

\textit{Estimates for $E_1(\mu_n)$}.
It follows from the asymptotics \eqref{asymptDelta}, \eqref{asymptDeltak} and Corollary~\ref{cor:DPK} that
\begin{align*}
& |\Delta(\mu_n)| \le C |\nu_n|^{-m}, \quad
|\Delta_1(\mu_n)| \le C |\nu_n|^{-(m-1)}, \\
& |\Delta^K(\mu_n)| \le C |\nu_n|^{-(m-1)}, \quad
|\Delta^K_1(\mu_n)| \le C |\nu_n|^{-(m-2)}
\end{align*}
for $n \in \mathbb Z$. Hence \eqref{defE1} implies
\begin{equation} \label{estE1}
|E_1(\mu_n)| \le C |\nu_n|^{-2(m-1)}, \quad n \in \mathbb Z.
\end{equation}

Applying Lemma~\ref{lem:PW} to the Paley-Wiener functions $\rho^{m+1}\hat \Delta(\rho^2)$ and $\rho^m \hat \Delta_1(\rho^2)$ and taking \eqref{esteps} into account, we get
$$
\| \{ \nu_n^{m+1} \hat \Delta(\mu_n) \}_{-\infty}^{\infty} \|_{l_2} \le C \epsilon, \quad \| \{ \nu_n^m \hat \Delta_1(\mu_n) \}_{-\infty}^{\infty} \|_{l_2} \le C \epsilon.
$$
Analogously, using Corollary~\ref{cor:stabDPK}, we estimate $\hat \Delta^K(\mu_n)$ and $\hat \Delta_1^K(\mu_n)$. Finally, we get
\begin{equation} \label{estE1n}
\| \{ \nu_n^{-(2m-1)}\hat E_1(\mu_n) \}_{-\infty}^{\infty} \|_{l_2} \le C \epsilon.
\end{equation}

\smallskip

Combining \eqref{relDn}, \eqref{estE}, \eqref{estEn}, \eqref{estE1}, and \eqref{estE1n}, we arrive at the estimate
$$
\| \{ \widehat{\mathscr D}_n - \widehat{\tilde{\mathscr D}}_n \}_{-\infty}^{\infty} \|_{l_2} \le C \epsilon.
$$
Thus \eqref{twoside} implies 
$$
\| \mathscr D - \tilde{\mathscr D} \|_{L_2(-1,1)} \le C \epsilon,
$$
so
$\| \hat D \|_{L_2(0,1)} \le C \epsilon$. Analogously, we obtain the estimate $\| \hat K \|_{L_2(0,1)} \le C \epsilon$,
which concludes the proof.
\end{proof}

Clearly, in Auxiliary Problem~\ref{ap:loop} and in Lemma~\ref{lem:stabaux}, the function $\Delta_1(\la)$ can be replaced by $\Delta_k(\la)$ for every $k \in \{ 2, 3, \dots, m \}$.

\begin{lem} \label{lem:stabla}
Let $h(\la)$ be any function of form \eqref{inth} with simple zeros $\{ \la_n \}_1^{\infty}$ and $h(0) \ne 0$. Then, there exists $\eps > 0$ such that, for any $\tilde K \in L_2(0,1)$ satisfying 
$$
\int_0^1 \tilde K(t) \, dt = 0, \quad \| \hat K \|_{L_2(0,1)} \le \eps,
$$
the function 
\begin{equation} \label{intht}
\tilde h(\la) := \frac{\sin \rho}{\rho} + \frac{1}{\rho^2} \int_0^1 \tilde K(t) \cos \rho t \, dt
\end{equation}
has only simple zeros $\{ \tilde \la_n \}_1^{\infty}$, and there holds
\begin{equation} \label{estrho}
\| \{ n \hat \rho_n \}_1^{\infty} \|_{l_2} \le C \| \hat K \|_{L_2(0,1)},
\end{equation}
where $\rho_n := \sqrt{\la_n}$, $\tilde \rho_n := \sqrt{\tilde \la_n}$, $\arg \rho_n, \arg \tilde \rho_n \in (-\tfrac{\pi}{2}, \tfrac{\pi}{2}]$, and the constant $C$ depends only on $K$.
\end{lem}

\begin{proof}
Let $\rho_n$ be a zero of $h(\rho^2)$. Then, there exists a contour $\gamma_n := \{ \rho \in \mathbb C \colon |\rho - \rho_n| = r \}$ ($r > 0$) that does not contain any other zeros $\rho_k \ne \rho_n$ and $\rho = 0$ inside of it. Consequently, there exists $\eps > 0$ such that, for
$\| \hat K \|_{L_2(0,1)} < \eps$, there holds
$$
\frac{|h(\rho^2) - \tilde h(\rho^2)|}{|h(\rho^2)|} < 1, \quad \rho \in \gamma_n.
$$
Rouche's Theorem implies that $\tilde h(\rho^2)$ has the single zero $\tilde \rho_n$ inside $\ga_n$. Moreover, $\tilde \rho_n \to \rho_n$ as $\eps \to 0$.

Obviously, we have
\begin{equation} \label{sm2}
h(\la_n) - \tilde h(\la_n) = \tilde h(\tilde \la_n) - \tilde h(\la_n).
\end{equation}
It follows from \eqref{inth} and \eqref{intht} that
\begin{equation} \label{sm3}
h(\la_n) - \tilde h(\la_n) = \frac{1}{\rho_n^2} \int_0^1 \hat K(t) \cos \rho_n t \, dt.
\end{equation}
Taylor's formula implies
\begin{equation} \label{sm4}
\tilde h(\tilde \la_n) - \tilde h(\la_n) = \tfrac{d}{d\rho} \tilde h(\rho_n^2)(\tilde \rho_n - \rho_n) + O\bigl((\rho_n - \tilde \rho_n)^2\bigr), \quad \eps \to 0.
\end{equation}
Since $\rho_n$ is simple, then there holds $\tfrac{d}{d\rho} h(\rho_n^2) \ne 0$ and so $\tfrac{d}{d\rho} \tilde h(\rho_n^2) \ne 0$ for sufficiently small $\eps$. Therefore, combining \eqref{sm2}, \eqref{sm3}, and \eqref{sm4}, we obtain
\begin{equation} \label{rhon1}
|\hat \rho_n| \le C \| \hat  K \|_{L_2(0,1)},
\end{equation}
where $C$ depends on $K$ and $n$. 

According to the proof of \cite[Lemma~6.3]{Bond20}, there exists $n_1 \in \mathbb N$ such that
$$
\sqrt{\sum_{n = n_1}^{\infty} (n |\hat \rho_n|)^2} \le C \| \hat K \|_{L_2(0,1)},
$$
where $n_1$ and $C$ depend only on $K$. Combining the latter estimate with \eqref{rhon1} for $n = \overline{1,n_1-1}$, we arrive at \eqref{estrho}.
\end{proof}

Thus, Lemmas~\ref{lem:stabaux} and~\ref{lem:stabla} together imply the stability of Auxiliary Problem~\ref{ap:loop}.
We point out that Lemma~\ref{lem:stabaux} lacks a solvability assertion, since the solution of Auxiliary Problem~\ref{ap:loop} may not exist for an arbitrarily small perturbation of the given data. Indeed, not for every perturbation of $\Delta(\la)$ and $\Delta_1(\la)$, the solution $d(\la)$ and $h(\la)$ of the system \eqref{defDelta}--\eqref{defDeltak} is entire in $\la$. In general, Auxiliary Problem~\ref{ap:loop} is overdetermined. In other words, under the assumption that the potentials $\{ q_j \}_1^m$ are known a priori, the potential $q_0$ can be recovered from a smaller amount of data than $\Delta(\la)$, $\Delta_1(\la)$, and the signs $\{ \sigma_n \}_1^{\infty}$. For example, it has been shown in \cite{BS20} that, for a slightly different Sturm-Liouville operator (in fact, for a pencil) on the graph $G$ with equal edge lengths, a fractional $\frac{2}{m+1}$-th part of the spectrum $\Lambda$ and a suitable sequence of signs are sufficient for the unique reconstruction of $q_0$. Alternatively, the proof of Lemma~\ref{lem:stabaux} 
shows that the data $\{ \Delta(\mu_n), \Delta_1(\mu_n) \}_{-\infty}^{\infty}$ and $\{ q_j \}_1^m$ uniquely specify the functions $D$ and $K$. Moreover, the corresponding problem is solvable for sufficiently small perturbations $\{ \nu_n^{m+1} \hat \Delta(\mu_n) \}_{-\infty}^{\infty}$ and $\{ \nu_n^m \hat \Delta_1(\mu_n) \}_{-\infty}^{\infty}$ in $l_2$ and
$\hat q_j$ in $L_2(0,T_j)$ ($j = \overline{1,m}$). 

\section{Quasiperiodic problem} \label{sec:quasi}

Denote by $\mathscr L := \mathscr L(q_0, a)$ the boundary value problem for equation \eqref{eqv} ($j = 0$, $T_0 = 1$) with the boundary conditions \eqref{bcquasi}. We call the Hill-type discriminant $d(\la)$ \eqref{defd}, the Dirichlet spectrum $\{ \la_n \}_1^{\infty}$, and the sequence of signs $\{ \sigma_n \}_1^{\infty}$ \eqref{defsi} the spectral data of the problem $\mathscr L$.

This section is focused on the following quasiperiodic inverse Sturm-Liouville problem on the loop, which arises at step~3 of Algorithm~\ref{alg:main}.

\begin{ip} \label{ip:quasi}
Given the spectral data $d(\la)$, $\{ \la_n \}_1^{\infty}$, $\{ \sigma_n \}_1^{\infty}$, find the potential $q_0$.
\end{ip}

The uniqueness of solution for Inverse Problem~\ref{ip:quasi} actually follows from the results of \cite{Plak88}. Constructive procedures for solving inverse spectral problems, more general than Inverse Problem~\ref{ip:quasi}, have been developed by Yurko in \cite{Yur08, Yur12, Yur16-caot, Yur20} and other studies. Here, we need the local stability of the quasiperiodic inverse problem from \cite{Bond25-quasi}.

\begin{prop} [\hspace*{-3pt}\cite{Bond25-quasi}, Theorem 2.5] \label{prop:quasi}
Suppose that $q_0 \in L^0_{2,\mathbb R}(0,1)$ and $a \in \mathbb R \setminus \{ -1, 0, 1 \}$ are fixed. Let $d(\la)$, $\{ \la_n \}_1^{\infty}$, $\{ \sigma_n \}_1^{\infty}$ be the spectral data of the problem $\mathscr L(q_0,a)$. Let $D(t)$ be the $L_2$-function from the relation \eqref{intd} and $\rho_n = \sqrt{\la_n}$, $\arg \rho_n \in (-\tfrac{\pi}{2}, \frac{\pi}{2}]$, $n \ge 1$. Then, there exists $\eps > 0$ (depending on $q_0$ and $a$) such that, for any $\tilde D \in L_{2,\mathbb R}(0,1)$ and any reals $\{ \tilde \rho_n^2 \}_1^{\infty}$ satisfying the inequalities
$$
    \| \hat D \|_{L_2(0,1)} \le \eps, \quad \| \{ n \hat \rho_n \}_1^{\infty} \|_{l_2} \le \eps
$$
and the condition
\begin{equation} \label{condsi}
\tilde d(\tilde \rho_n^2) = \pm 2 \quad \text{iff} \quad \sigma_n = 0,
\end{equation}
there exists a unique $\tilde q_0 \in L^0_{2,\mathbb R}(0,1)$ such that
$\tilde d(\la)$, $\{ \tilde \la_n \}_1^{\infty}$, $\{ \sigma_n \}_1^{\infty}$ are the spectral data of $\mathscr L(\tilde q_0, a)$, where
$$
\tilde d(\la) := (a + a^{-1}) \cos \rho + \frac{1}{\rho^2} \int_0^1 \tilde D(t) \sin \rho t \, dt, \quad \tilde \la_n := \tilde \rho_n^2.
$$
Moreover, 
\begin{equation} \label{estq0}
\| \hat q_0 \|_{L_2(0,1)} \le C \left( \| \hat D \|_{L_2(0,1)} + \| \{ n \hat \rho_n \}_1^{\infty} \|_{l_2} \right),
\end{equation}
where the constant $C$ depends only on $q_0$ and $a$.
\end{prop}

Proposition~\ref{prop:quasi} together with the results of Sections~\ref{sec:bound} and~\ref{sec:transit} imply the stability of recovering the potential $q_0$ and so allow us to complete the proof of the main theorem.

\begin{proof}[Proof of Theorem~\ref{thm:loc}]
Suppose that the problem $\mathcal L(\mathbf{q})$ and the potentials $\tilde{\mathbf{q}}$ satisfy the hypothesis of Theorem~\ref{thm:loc} for some $\eps > 0$. If $\eps$ is small enough, then the estimates \eqref{estqj} for $j = \overline{1,m}$ follow from Theorem~\ref{thm:GLloc}. Next, subsequently applying Lemmas~\ref{lem:stabaux} and~\ref{lem:stabla}, we arrive at the estimate
$$
\|\hat D \|_{L_2(0,1)} + \| \{ n \hat \rho_n \}_1^{\infty} \|_{l_2} \le C \left( \| \rho^m \hat \Delta(\rho^2) \|_{L_2(\mathbb R)} + \| \rho^m \hat \Delta_1(\rho^2) \|_{L_2(\mathbb R)} + \sum_{j = 1}^m \| \hat q_j \|_{L_2(0,T_j)} \right) \le C \de,
$$
where $\de$ is defined by \eqref{difDe}. Consequently, for sufficiently small $\eps$, the potential $\tilde q_0$ fulfills the conditions of Proposition~\ref{prop:quasi}. In particular, the condition \eqref{condsi} is necessary for the quasiperiodic problem $\mathscr L(\tilde q_0, a)$ by virtue of \cite[Theorem~2.2]{Bond25-quasi} and the hypothesis $\sigma_n = \tilde \sigma_n$. Applying Proposition~\ref{prop:quasi}, we get the estimate \eqref{estq0}, which concludes the proof.
\end{proof}

Note that Lemma~\ref{lem:stabla} requires the additional condition $h(0) \ne 0$. Nevertheless, that condition can be removed by a technical modification and does not influence the final result \eqref{estqj}.

\section{Uniform stability of the inverse problem} \label{sec:uni}

In this section, we discuss the uniform stability of Inverse Problem~\ref{ip:quasi} and, in particular, prove Theorem~\ref{thm:uni} for the boundary edges. Our proof does not require the Gelfand-Levitan-type equation \eqref{GL} and relies on the reconstruction formula of the method of spectral mappings~\cite{Yur02, Yur08}.

Let numbers $Q > 0$, $a \in \mathbb R \setminus \{ 0 \}$, and $\{ T_j \}_0^m$ be fixed. 
Consider any potentials $\mathbf{q}$ and $\tilde{\mathbf{q}}$ in $B_Q$.

\begin{proof}[Proof of Theorem~\ref{thm:uni}]
Fix $k \in \{ 1, 2, \dots, m \}$. For brevity, let us omit the index $k$ in such notations as $q := q_k$, $T := T_k$, etc., similarly to Section~\ref{sec:bound}.

Our proof relies on the following reconstruction formula of the method of spectral mappings (see \cite{Yur02, Yur08}):
\begin{equation} \label{recq}
q(x) = \tilde q(x) + \frac{1}{2\pi i} \int_{\Gamma} \bigl( S(x, \la) \tilde S(x, \la) \bigr)' \hat M(\la) \, d\la,
\end{equation}
where the contour $\Gamma$ is defined in \eqref{defGa}. Note that, by virtue of Corollary~\ref{cor:stabDelta}, the spectra $\Lambda$ and $\tilde \Lambda$ possess the same lower bound that depends only on $Q$. Therefore, one can choose the parameter $\tau > 0$ in \eqref{defga} that is large enough for every $\mathbf{q}$ and $\tilde{\mathbf{q}}$ in the ball $B_Q$.

By Lemma~\ref{lem:hatM}, we have $\hat M(\rho^2) \in L_2(\ga)$. Using the definition \eqref{defM}, the estimates \eqref{est2side}, which are uniform in the ball $B_Q$, and following the proof of Theorem~\ref{thm:GLloc}, we get the estimate \eqref{estM} with a constant $C$ that depends only on $Q$.

The both functions $S(x,\la)$ and $\tilde S(x,\la)$ satisfy the asymptotics \eqref{asymptS} and \eqref{asymptSp}, where the $O$-estimates are uniform in $B_Q$. Substituting those asymptotics into \eqref{recq} and using the relation $d\la = 2 \rho d\rho$, we derive
$$
\hat q(x) = I_1(x) + I_2(x), \quad I_1(x) := \frac{1}{\pi i} \int_{\ga} \sin 2 \rho x \hat M(\rho^2) \, d \rho, \quad
I_2(x) = \int_{\ga} O(\rho^{-1}) \hat M(\rho^2) \, d\rho,
$$
where $O(\rho^{-1})$ is actually a continuous function of $x \in [0,T]$. Clearly, $I_1(x)$ is the Fourier transform of the function $\hat M(\rho^2)$ satisfying \eqref{estM}. Hence
$$
I_1 \in L_2(0, T), \quad \| I_1 \|_{L_2(0,T)} \le C \de_k.
$$
In view of \eqref{estM}, the integral $I_2(x)$ converges absolutely and uniformly on $[0,T]$, so
$$
I_2 \in C[0,T], \quad |I_2(x)| \le C \de_k, \quad x \in [0,T].
$$
This together with \eqref{defdek} imply \eqref{quni}.
\end{proof}

Analogously to the proofs of Lemmas~\ref{lem:stabaux} and \ref{lem:stabla}, we can show the uniform stability of Auxiliary Problem~\ref{lem:stabaux} in the ball $B_Q$. However, there arise difficulties with the uniform stability of Inverse Problem~\ref{ip:quasi}. In \cite{Bond25-quasi}, the author has obtained the uniform stability of the quasiperiodic inverse problem under the conditions
\begin{equation} \label{reqd}
(-1)^n \sign a \cdot d(\la_n) 
\begin{cases}
\ge 2 + \eps & \text{if} \: \sigma_n \ne 0, \\
= 2 & \text{if} \: \sigma_n = 0,
\end{cases}
\quad n \ge 1,
\end{equation}
which are not guaranteed in the ball $\| q_0 \|_{L_2(0,1)} \le Q$. Furthermore, the author has not found any conditions in terms of the spectral data of the problem \eqref{eqv}--\eqref{mc} on the graph that make it possible to achieve \eqref{reqd} for the spectral data of the corresponding quasiperiodic problem. Thus, the question of the uniform stability for reconstruction of the potential $q_0$ on the loop remains open.

\medskip

{\bf Funding.} This work was supported by Grant 24-71-10003 of the Russian Science Foundation, https://rscf.ru/en/project/24-71-10003/.

\noindent Natalia Pavlovna Bondarenko \\

\noindent 1. Department of Mechanics and Mathematics, Saratov State University, 
Astrakhanskaya 83, Saratov 410012, Russia, \\

\noindent 2. Department of Applied Mathematics, Samara National Research University, \\
Moskovskoye Shosse 34, Samara 443086, Russia, \\

\noindent 3. S.M. Nikolskii Mathematical Institute, RUDN University, 6 Miklukho-Maklaya St, Moscow, 117198, Russia, \\

\noindent 4. Moscow Center of Fundamental and Applied Mathematics, Lomonosov Moscow State University, Moscow 119991, Russia.\\

\noindent e-mail: {\it bondarenkonp@sgu.ru}

\end{document}